\documentclass[11pt,letterpaper]{amsart}

\numberwithin{equation}{section}
\theoremstyle{plain}
\newtheorem{theorem}{Theorem}[section]

\newtheorem{lemma}[theorem]{Lemma}

\newtheorem{question}{Question}

\theoremstyle{definition}

\newtheorem{remark}[theorem]{Remark}

\newcommand{\Rmnum}[1]{\expandafter\@slowromancap\romannumeral #1@}

\newcommand{\mr}{\mathbb{R}}
\newcommand{\ud}{\mathrm{d}}
\newcommand{\ms}{\mathbb{S}}
\allowdisplaybreaks

\keywords{prescribed Q-curvature, complete metric, conformally flat.}
\subjclass{Primary: 53C18,   Secondary: 58J90.}
\address{Mingxiang Li, Department  of Mathematics \& Institue of Mathematical Sciences,  The Chinese University of Hong Kong, Shatin, NT, Hong Kong   }
\email{mingxiangli@cuhk.edu.hk}

\begin{document}
	\title{Obstructions to prescribed Q-curvature of complete  conformal  metrics on  $\mathbb{R}^n$  }
		\author{ Mingxiang Li}
\date{}
\maketitle
\begin{abstract}
	We  provide some obstructions to the prescribed Q-curvature problem for   the complete  conformal metrics  on $\mr^n$ with finite total Q-curvature. One of them   is a  Bonnet-Mayer type  theorem with respect to Q-curvature. Others  are related to  the decay rate of the  prescribed functions. 
\end{abstract}
	\section{Introduction}

	Given a smooth function $K(x)$ on standard sphere $(\ms^2, g_0)$,  the well-known Nirenberg problem is to find a conformal metric $g=e^{2u}g_0$ such that its Gaussian curvature equals to  $f$.  It is equivalent to solving the following conformally invariant equation on $\ms^2$
	\begin{equation}\label{Nirenberg problem}
		-\Delta_{\ms^2}u(x)+1=K(x)e^{2u(x)},\; x\in\ms^2
	\end{equation}
	where $\Delta_{\ms^2}$ is the Laplace-Beltrami operator. The famous Chern-Gauss-Bonnet formula requires that $\sup K>0$ which is an obvious obstruction. Surprisingly, another  obstruction to \eqref{Nirenberg problem} known as Kazdan-Warner identity \cite{KZ 74 Compact} was established  which can be stated as  follows
	\begin{equation}\label{KZ indentity}
		\int_{\ms^2}\langle\nabla x_i,\nabla K\rangle e^{2u}\ud\mu_{\ms^2}=0,\; 1\leq i\leq 3
	\end{equation}
where $x_i$ is the eigenfunction  satisfying $-\Delta_{\ms^2}x_i=2x_i$.
	Interested readers may refer to \cite{CY Acta}, \cite{CY JDG} for more information about Nirenberg problem.  A direct corollary  of the identity \eqref{KZ indentity} is that $f(x)=1+tx_i$ for any $t\not=0$ can not be the prescribed Gaussian curvature on $\ms^2$.  
	
	For open surfaces, without restricting in the conformal class,  some results have been  established in \cite{KZ 74 open} by  Kazdan and Warner. In particular, Theorem 4.1 in \cite{KZ 74 open} gives a necessary and sufficient condition for a smooth function on $\mr^2$ to be the prescribed Gaussian curvature of a complete Riemannian metric. However,  restricting in  the conformal class, the situation becomes very subtle.  Throughout this paper, we focus on the conformal  metrics of Euclidean space $\mr^n$ where $n\geq2$ is an even integer.  It  is better to  start from the two dimensional case. Given a smooth function $f(x)$ on $\mr^2$, we consider the following conformally invariant  equation
	\begin{equation}\label{NP for R^n}
		-\Delta u(x)=f(x)e^{2u(x)},\; x\in \mr^2.
	\end{equation}
Indeed, via a stereographic projection, the equation \eqref{Nirenberg problem} can be transformed into  \eqref{NP for R^n}.
	There are a lot of works devoted to this equation \eqref{NP for R^n} including \cite{CYZ}, \cite{CL 91 Duke}, \cite{CL 93}, \cite{Cheng-Lin MA}, \cite{HuTr},  \cite{Li23 MA}, \cite{McOwen}, \cite{Ni 82}, \cite{Sat}, \cite{Struwe} and many others.  	In particular, for   $f(x)\leq 0$, it has been well understood by the works  \cite{CYZ}, \cite{Cheng-Lin MA}, \cite{KY},  \cite{Ni 82},    \cite{Sat}  and many others.	
	
	In this paper, the  completeness of the metrics will be taken into account.
Under such geometric restriction, Cohn-Vossen  \cite{CV} and Huber \cite{Hu} gave a control of the Gaussian curvature integral. For readers' convenience, a baby version of  their results can be stated as follows.  Throughout this paper, $\varphi^+$ and $\varphi^-$  denote the positive part  and negative part of function $\varphi$ respectively.
	\begin{theorem}\label{thm: Cohn-Vossen inq}
		{\bf (Cohn-Vossen \cite{CV}, Huber \cite{Hu})} Consider a complete  metric $g=e^{2u}|dx|^2$ on $\mr^2$. If the negative part of its Gaussian curvature $K_g$ is integrable on $(\mr^2, e^{2u}|dx|^2)$ i.e.
		$$\int_{\mr^2}K_g^-e^{2u}\ud x<+\infty,  $$
		then there holds
		$$\int_{\mr^2} K_ge^{2u}\ud x\leq 2\pi.$$ 
	\end{theorem}
For higher dimensional cases $n\geq4$ and  a conformal metric $g=e^{2u}|dx|^2$ on $\mr^n$, the Q-curvature with respect to such metric satisfies  the following conformally invariant equation
\begin{equation}\label{Q curvature equa}
	(-\Delta)^{\frac{n}{2}}u(x)=Q_g(x)e^{nu(x)},\; x\in \mr^n.
\end{equation} 
We say that the conformal metric $g=e^{2u}|dx|^2$  on $\mr^n$ has  a finite total Q-curvature if 
$$\int_{\mr^n}|Q_g|e^{nu}\ud x<+\infty.$$
Similar to two dimensional case, the equation \eqref{Q curvature equa} also comes from the standard sphere through  a stereographic projection.
Concerning the prescribed Q-curvature on standard sphere $\ms^n$, one may refer to \cite{Brendle}, \cite{CY 95 Ann}, \cite{JLX}, \cite{MalStru}, \cite{WX 09 JFA}  for more details. From analytic point of view to study the equation \eqref{Q curvature equa}, interested readers may  refer to \cite{HMM}, \cite{Li23 MA}, \cite{Lin}, \cite{Mar MZ},  \cite{WX} for more information. From geometric point of view, similar to Theorem \ref{thm: Cohn-Vossen inq}, the Q-curvature integral  is bounded from above under suitable geometric assumptions.
\begin{theorem}\label{thm: CQY, F,  NX}
	{\bf (Chang-Qing-Yang \cite{CQY}, Fang \cite{Fa}, Ndiaye-Xiao \cite{NX})} Consider a complete conformal  metric $g=e^{2u}|dx|^2$ on $\mr^n$ where $n\geq 4$ is an even integer with finite total Q-curvature. If the scalar curvature $R_g\geq 0$ near infinity, there holds
	$$\int_{\mr^n}Q_ge^{nu}\ud x\leq \frac{(n-1)!|\mathbb{S}^n|}{2}$$
	where $|\ms^n|$ denotes the volume of standard sphere $\ms^n$. 
\end{theorem}
For more related results, interested readers may refer to \cite{CQY2},  \cite{Li 23 Q-curvature}, \cite{Lu-Wang}, \cite{Wang Yi} and the references therein.

Recalling the equation \eqref{Q curvature equa}, a natural question is that  what kind of prescribed functions $f(x)$ on $\mr^n$ we can find a complete conformal metric $g=e^{2u}|dx|^2$ on $\mr^n$ such that $Q_g=f.$  Does it have  obstructions  like Kazdan-Warner identity \eqref{KZ indentity}?  To explore such a question, it is better to start with 
the  famous  Bonnet-Mayer's theorem which shows that for a complete manifold $(M^n, g_b)$, if the Ricci curvature $Ric_{g_b}\geq (n-1)g_b$, then $(M, g_b)$ is compact.  Interested readers may refer to Chapter 6 of \cite{Petersen} for more details. With help of Bonnet-Mayer's theorem,  an obvious obstruction  occurs for $n=2$.
	 \begin{theorem}\label{thm: Mayer's theorem}
	 	{\bf (Bonnet-Mayer's theorem)} 
	 	Given  a smooth function $f(x)\geq 1$ on $\mr^2$.
	 	There is no complete conformal metric $g=e^{2u}|dx|^2$ on $\mr^2$ such that its Gaussian curvature $K_g=f$.
	 \end{theorem}

Firstly, inspired by such a  result,  we  generalize it  to all higher dimensional cases.
\begin{theorem}\label{thm: f geq 1}
	Given a  smooth  function $f(x)$ on $\mr^n$ where $n\geq 2$ is an even integer and $f(x)\geq 1$ near infinity,   there is no complete  conformal metric $g=e^{2u}|dx|^2$ on $\mr^n$ with finite total Q-curvature such that its  Q-curvature $Q_g=f$.
\end{theorem}

One may ask whether $f\geq 1$ near infinity  is a sharp barrier for the existence of   complete conformal  metrics and what kind of behaviors occur if $f(x)$ tends to zero near infinity. Precisely, consider a function $f$ satisfying 
\begin{equation}\label{s decay}
	|f(x)|\leq C(|x|+1)^{-s}, \; s>0
\end{equation}
where $C$ is a positive constant which may be different from line to line throughout this paper. When $f(x)$ is positive somewhere and satisfies \eqref{s decay},
the existence of solutions to the following equation
\begin{equation}\label{conformal equation}
	(-\Delta )^{\frac{n}{2}}u(x)=f(x)e^{nu(x)}, \; x\in \mr^n
\end{equation}
has been established by Theorem 1  in \cite{McOwen} for $n=2$ and Theorem 2.1 in \cite{Chang-Chen} for $n\geq 4$. Taking the  completeness of   metrics into account, Aviles \cite{Aviles} studied the equation \eqref{conformal equation} for $n=2$ and he showed that,  for   $f$ positive somewhere and  $s\geq 2$ in \eqref{s decay}, there exists  complete conformal metric. Besides, for $0<s<1$, if  $f$ satisfies 
\begin{equation}\label{f|x|^s=1}
	\lim_{|x|\to\infty} f(x)|x|^s=1,
\end{equation}
Aviles claimed  that there also exists a complete metric(See Theorem $\mathrm{A^1}$ in \cite{Aviles}). However, Cheng and Lin constructed a family of functions $f(x)$ satisfying \eqref{f|x|^s=1} (See Theorem 1.1 \cite{Cheng-Lin Pisa}) to show the non-existence of   complete   conformal metric which  contradicts to  Aviles's claim.  Cheng and Lin's example tells  us that the complete conformal metric exists some obstructions even the given function $f$ satisfies \eqref{s decay}. Inspired by this, we will give another barrier. 

In fact,  Kazdan-Warner identity \eqref{KZ indentity} establishes an obstruction for the prescribed Q-curvature on $\ms^n$, one could ask whether there are  some  barriers from this perspective.               Indeed, Kazdan-Warner identity on $\mr^n$ is known as  Pohozaev's identity, and several works, including \cite{CL 93}, \cite{Cheng-Lin MA}, \cite{Li23 MA}, \cite{LW}, \cite{Xu05}, and many others, are devoted to it.   We provide an obstruction  for the existence of complete conformal  metric from this point of view.
	\begin{theorem}\label{thm:xnabla log f geq -1}
	Given a  positive and smooth $f(x)$ on $\mr^n$ where $n\geq 2$ is an even integer and $f(x)$ satisfies 
	\begin{equation}\label{condition for f}
		\frac{x\cdot \nabla f(x)}{f(x)}\geq -\frac{n}{2}.
	\end{equation}
	Then there is no complete conformal metric $g=e^{2u}|dx|^2$ on $\mr^n$  with finite total Q-curvature  such that its  Q-curvature $Q_g=f(x)$.
\end{theorem}
\begin{remark}
	We will show that  the condition \eqref{condition for f} is sharp  to some degree in Section \ref{section:sharp}.
\end{remark}
For non-positive functions, we also obtain a barrier.
\begin{theorem}\label{thm: f leq 0}
	Given a  non-positive and smooth $f(x)$ on $\mr^n$ where $n\geq 2$ is an even integer and $f(x)$ satisfies 
	\begin{equation}\label{f leq -|x|^-n}
		f(x)\leq -C|x|^{-n},\;|x|\gg1.
	\end{equation}
	Then there is no complete   conformal metric $g=e^{2u}|dx|^2$ on $\mr^n$  with finite total Q-curvature  such that its  Q-curvature $Q_g=f(x)$.
\end{theorem}
\begin{remark}
	In fact, without completeness, the conclusion still holds for $n=2$ by the result of Sattinger \cite{Sat}. However, for $n\geq 4$ and  $f(x)\equiv-1$, a result of Martinazzi \cite{Martinazzi AANL} showed the existence of  non-complete  conformal metric.
\end{remark}
Furthermore, if  prescribed  functions may change sign, another barrier is established.
\begin{theorem}\label{thm: f < 0 near infinity}
	Given a  smooth $f(x)$ on $\mr^n$ where $n\geq 2$ is an even integer and $f(x)$ satisfies 
	\begin{equation}\label{f leq- |x|^s}
		f(x)\leq -C|x|^{s},\;|x|\gg1\;\mathrm{and}\; s>0.
	\end{equation}
	Then there is no complete   conformal metric $g=e^{2u}|dx|^2$ on $\mr^n$  with finite total Q-curvature  such that its  Q-curvature $Q_g=f(x)$.
\end{theorem}
Now, we briefly introduce the structure of this paper. In Section \ref{section: integral estimate}, some  results established in \cite{Li 23 Q-curvature} are reviewed for later use.  Subsequently,
we prove Theorem \ref{thm: f geq 1}, Theorem \ref{thm:xnabla log f geq -1}, Theorem \ref{thm: f leq 0} and Theorem \ref{thm: f < 0 near infinity} in Section \ref{section:proof}. Finally, the sharpness of  condition \eqref{condition for f} is discussed.

{\bf Acknowledgment.} The author would like to thank Professor Xingwang Xu, Dong Ye, Xia Huang, Biao Ma for helpful discussions. The author is also deeply grateful to the referee for their valuable suggestions and comments.

\section{Integral estimates and Pohozaev's identity}\label{section: integral estimate}

\begin{lemma}\label{lem: f geq |x|^s}
Given a positive and smooth function $f(x)$ on $\mr^n$.  Supposing  that,    for $|x|\gg1$ and $s\in\mr$, 
$$\frac{x\cdot \nabla f(x)}{f(x)}\geq s,$$
  there holds
$$ f(x)\geq c_0|x|^{s}, \; |x|\gg1$$
where $c_0$ is a positive constant.
\end{lemma}
\begin{proof}
	Based on our assumption, there exists $t_1>0$ such that, for $|x|\geq t_1$,
	$$\frac{x\cdot \nabla f(x)}{f(x)}\geq s.$$
With help of such estimate,   for $|x|>t_1$,  one has 
	\begin{align*}
		\log f(x)-\log f(\frac{t_1}{|x|}x)=&\int_{t_1}^{|x|}\frac{x}{|x|}\cdot\nabla \log f(t\frac{x}{|x|})\ud t\\
		=&\int_{t_1}^{|x|}\left(t\frac{x}{|x|}\cdot\nabla \log f(t\frac{x}{|x|})\right)\frac{1}{t}\ud t\\
		\geq & \int_{t_1}^{|x|}s\frac{1}{t}\ud t\\
		=& s\log |x|-s\log t_1
	\end{align*}
which yields that
$$f(x)\geq t_1^{-s}\left(\min_{|y|=t_1}f(y)\right)|x|^s, \; |x|>t_1.$$
Thus, we finish our proof.
\end{proof}

Recall  the conformally invariant equation
\begin{equation}\label{higher order equation}
	(-\Delta)^{\frac{n}{2}}u(x)=f(x)e^{nu(x)},\; x\in\mr^n.
\end{equation}
We say the solution to \eqref{higher order equation} is  normal  if $u$ satisfies the integral equation
\begin{equation}\label{normal solution}
	u(x)=\frac{2}{(n-1)!|\mathbb{S}^n|}\int_{\mr^n}\log\frac{|y|}{|x-y|}f(y)e^{nu(y)}\ud y+C_0
\end{equation}
where $C_0$ is a constant.
For  more details about normal solutions, one may refer to Section 2 of \cite{Li 23 Q-curvature}.

Given a function $\varphi(x)\in L^\infty_{loc}(\mr^n)\cap L^1(\mr^n)$, we can define the logarithmic potential
$$\mathcal{L}(\varphi)(x):=\frac{2}{(n-1)!|\mathbb{S}^n|}\int_{\mr^n}\log\frac{|y|}{|x-y|}\varphi(y)\ud y.$$
For brevity, we set  the notation  $\alpha$ defined as
$$\alpha:=\frac{2}{(n-1)!|\mathbb{S}^n|}\int_{\mr^n}\varphi(y)\ud y.$$
Meanwhile,  $B_r(p)$ denotes the Euclidean ball with radius $r$ centered at $p\in \mr^n$ and $|B_r(p)|$ denotes its volume respect to standard Euclidean metric.

The following lemmas related to the properties of $\mathcal{L}(\varphi)$ have  been established in \cite{Li 23 Q-curvature} and we repeat the proofs for readers' convenience.

\begin{lemma} \label{lem: L(f)}
	For $|x|\gg1$, there holds
	\begin{equation}\label{Lf=-aplha log x+}
		\mathcal{L}(\varphi)(x)=(-\alpha+o(1))\log|x|+\frac{2}{(n-1)!|\mathbb{S}^n|}\int_{B_1(x)}\log\frac{1}{|x-y|}\varphi(y)\ud y
	\end{equation}
	where  $o(1)\to 0$ as $|x|\to\infty$.
\end{lemma}
\begin{proof}
	
	Choose $|x|\geq e^4$ such  that $|x|\geq 2\log|x|$. Split $\mr^n$ into three pieces
	$$A_1=B_{1}(x), \quad A_2=B_{\log|x|}(0),\quad A_3=\mr^n\backslash (A_1\cup A_2).$$
	For $y\in A_2$ and $|y|\geq2$ , we have $|\log\frac{|x|\cdot|y|}{|x-y|}|\leq \log(2\log|x|)$.
	Respectively,  for $|y|\leq 2$, $ |\log\frac{|x|\cdot|y|}{|x-y|}|\leq |\log|y||+C$.
	Thus 
	\begin{equation}\label{A_2}
		|\int_{A_2}\log\frac{|y|}{|x-y|}\varphi(y)\ud y+\log|x|\int_{A_2}\varphi(y)\ud y|\leq C\log\log|x|+C=o(1)\log|x|.
	\end{equation}
	For $y\in A_3$, it is not hard to check 
	$$\frac{1}{|x|+1}\leq\frac{|y|}{|x-y|}\leq |x|+1.$$ With help  of this estimate, we could  control the integral over $A_3$ as
	\begin{equation}\label{A_3}
		|\int_{A_3}\log\frac{|y|}{|x-y|}\varphi(y)\ud y|\leq \log(|x|+1)\int_{A_3}|\varphi|\ud y.
	\end{equation}
	For $y\in B_1(x)$, one has
	$1\leq |y|\leq |x|+1$ and then
	$$|\int_{A_1}\log|y|\varphi\ud y|\leq \log(|x|+1)\int_{A_1}|\varphi|\ud y.$$
	Since $f\in L^1(\mr^n)$, notice that $\int_{A_3\cup A_1}|\varphi|\ud y\to 0$ as $|x|\to \infty$ and 
	$$\frac{2}{(n-1)!|\mathbb{S}^n|}\int_{A_2}\varphi(y)\ud y=\alpha+o(1).$$
	Thus there holds
	\begin{equation}\label{u =-aplha log x1}
		\mathcal{L}(\varphi)(x)=(-\alpha+o(1))\log|x|+\frac{2}{(n-1)!|\mathbb{S}^n|}\int_{B_1(x)}\log\frac{1}{|x-y|}\varphi(y)\ud y.
	\end{equation}
\end{proof}

\begin{lemma}\label{lem:B_r_0|x|}
	For  $0<r_1<1$ fixed and $|x|\gg1$, there holds
	\begin{equation}\label{B_|x|/2v}
		\frac{1}{|B_{r_1|x|}(x)|}\int_{B_{r_1|x|}(x)}\mathcal{L}(\varphi)(y)\ud y=(-\alpha +o(1))\log|x|.
	\end{equation}
\end{lemma}
\begin{proof}
	By a direct computation and Fubini's theorem, one has 
	\begin{align*}
		&\int_{B_{r_1|x|}(x)}|\int_{B_1(z)}\log\frac{1}{|z-y|}\varphi(y)\ud y|\ud z\\
		\leq & 	\int_{B_{r_1|x|}(x)}\int_{B_1(z)}\frac{1}{|z-y|}|\varphi(y)|\ud y\ud z\\
		\leq &\int_{B_{r_1|x|}(x)}\int_{B_{r_1|x|+1}(x)}\frac{1}{|z-y|}|\varphi(y)|\ud y\ud z\\
		\leq &\int_{B_{r_1|x|+1}(x)}|\varphi(y)|\int_{B_{r_1|x|}(x)}\frac{1}{|z-y|}\ud z\ud y\\
		\leq & \int_{B_{r_1|x|+1}(x)}|\varphi(y)|\int_{B_{2r_1|x|+1}(0)}\frac{1}{|z|}\ud z\ud y\\
		\leq &C|x|^{n-1}.
	\end{align*}
	Thus 
	\begin{equation}\label{B_r_0|x|Q}
		\frac{1}{|B_{r_1|x|}(x)|}\int_{B_{r_1|x|}(x)}\int_{B_1(z)}\log\frac{1}{|z-y|}\varphi(y)\ud y\ud z=O(|x|^{-1}).
	\end{equation}
	Meanwhile, for  $y\in B_{r_1|x|}(x)$, there holds
	\begin{equation}\label{log|x|/|x_0|}
		|\log\frac{|y|}{|x|}|\leq \log\frac{1}{1-r_1}+\log(1+r_1)\leq C.
	\end{equation}
	With help of these estimates \eqref{B_r_0|x|Q}, \eqref{log|x|/|x_0|} and 
	Lemma \ref{lem: L(f)},  we have
	\begin{equation}\label{int B_|x|/2 v}
		\frac{1}{|B_{r_1|x|}(x)|}\int_{B_{r_1|x|}(x)}\mathcal{L}(\varphi)(y)\ud y=(-\alpha+o(1))\log|x|.
	\end{equation}
\end{proof}

\begin{lemma}\label{lem: varphi >0}
If $\varphi\geq 0$ near infinity, for $|x|\gg1$, there holds
	$$\mathcal{L}(\varphi)(x)\geq -\alpha\log|x|-C.$$
\end{lemma}
\begin{proof}
	By a direct computation, we have
\begin{align*}
	&\frac{(n-1)!|\mathbb{S}^n|}{2}(\mathcal{L}(\varphi)(x)+\alpha\log|x|)\\
	=&\int_{\mr^n}\log\frac{|x|\cdot(|y|+1)}{|x-y|}\varphi(y)\ud y+\int_{\mr^n}\log\frac{|y|}{|y|+1}\varphi(y)\ud y\\
	=&\int_{\mr^n}\log\frac{|x|\cdot(|y|+1)}{|x-y|}\varphi^+(y)\ud y-\int_{\mr^n}\log\frac{|x|\cdot(|y|+1)}{|x-y|}\varphi^-(y)\ud y\\
	&+\int_{\mr^n}\log\frac{|y|}{|y|+1}\varphi(y)\ud y.
\end{align*}
For $|x|\geq 1 $, it is easy to check that 
$$\frac{|x|\cdot(|y|+1)}{|x-y|}\geq 1$$
which shows that 
$$\log\frac{|x|\cdot(|y|+1)}{|x-y|}\geq 0.$$
Immediately, one has
$$\int_{\mr^n}\log\frac{|x|\cdot(|y|+1)}{|x-y|}\varphi^+(y)\ud y\geq 0.$$

Based on our assumption, $\varphi^-$ has compact support, there exists  $R_1>0$ such that $supp(\varphi^-)\subset B_{R_1}(0)$.
And for $|x|\geq 2R_1$, we have
\begin{align*}
	&\int_{\mr^n}\log\frac{|x|\cdot(|y|+1)}{|x-y|}\varphi^-(y)\ud y\\
	=&\int_{B_{R_1}(0)}\log\frac{|x|\cdot(|y|+1)}{|x-y|}\varphi^-(y)\ud y\\
	\leq &\log\left(2|R_1|+2\right)\int_{B_{R_1}(0)}\varphi^-(y)\ud y\\
	\leq &C
\end{align*}
where we use the fact
$\frac{|x|}{|x-y|}\leq 2$ for $|x|\geq 2R_1$ and $y\in B_{R_1}(0)$.

Since $\varphi \in L^\infty_{loc}(\mr^n)\cap L^1(\mr^n)$,  one has
\begin{align*}
	|\int_{\mr^n}\log\frac{|y|}{|y|+1}\varphi(y)\ud y|\leq &|\int_{B_2(0)}\log\frac{|y|}{|y|+1}\varphi(y)\ud y|+|\int_{\mr^n\backslash B_2(0)}\log\frac{|y|}{|y|+1}\varphi(y)\ud y|\\
	\leq & C\int_{B_2(0)}|\log|y||\ud y+C\int_{B_2(0)}|\log(|y|+1)|\ud y\\
	&+\log\frac{3}{2}\int_{\mr^n\backslash B_2(0)}|\varphi(y)|\ud y\\
	\leq &C.
\end{align*}
Combining these estimates, for $|x|\gg1$, we obtain that
$$\mathcal{L}(\varphi)(x)\geq -\alpha\log|x|-C.$$

\end{proof}

\begin{lemma}\label{lem: Delta Lf}
	For $R\gg1$, there holds
	$$\int_{B_R(0)}|\mathcal{L}(\varphi)(x)|\ud x=O((\log R)\cdot R^n).$$
\end{lemma}
\begin{proof}
	A direct computation and 	 Fubini's theorem yield that
	\begin{align*}
		&\int_{B_R(0)}|\mathcal{L}(\varphi )|\ud x\\
		\leq &C\int_{B_R(0)}\int_{\mr^n\backslash B_{2R}(0)}|\log\frac{ |y|}{|x-y|}|\cdot|\varphi(y)|\ud y\ud x\\
		&+C\int_{B_R(0)}\int_{ B_{2R}(0)}|\log\frac{ |y|}{|x-y|}|\cdot|\varphi(y)|\ud y\ud x\\
		\leq &C\int_{B_R(0)}\int_{\mr^n\backslash B_{2R}(0)}|\log\frac{ |y|}{|x-y|}|\cdot|\varphi(y)|\ud y\ud x\\
		&+C\int_{B_R(0)}\int_{ B_{2R}(0)}|\log|y||\cdot|\varphi(y)|\ud y\ud x\\
		&+C\int_{B_R(0)}\int_{ B_{2R}(0)}|\log|x-y||\cdot|\varphi(y)|\ud y\ud x.
	\end{align*}
We deal with these three terms one by one.
For $|x|\leq R$ and $y\in \mr^n\backslash B_{2R}(0)$, it is easy to verify that 
 $$\frac{1}{2}\leq \frac{|y|}{|x-y|}\leq 2.$$
  With help of this  fact, the first term can be controlled as follows
  \begin{align*}
  &	\int_{B_R(0)}\int_{\mr^n\backslash B_{2R}(0)}|\log\frac{ |y|}{|x-y|}|\cdot|\varphi(y)|\ud y\ud x\\
  \leq & \log 2\int_{B_R(0)}\int_{\mr^n\backslash B_{2R}(0)}|\varphi(y)|\ud y\ud x\\
  \leq&  CR^n.
  \end{align*}

As for the second term, one has
\begin{align*}
	&\int_{B_R(0)}\int_{ B_{2R}(0)}|\log|y||\cdot|\varphi(y)|\ud y\ud x\\
	\leq &\int_{B_R(0)}\int_{ B_{1}(0)}|\log|y||\cdot|\varphi(y)|\ud y\ud x\\
	&+\int_{B_R(0)}\int_{ B_{2R}(0)\backslash B_1(0)}|\log|y||\cdot|\varphi(y)|\ud y\ud x\\
	\leq &CR^n\int_{ B_{1}(0)}|\log|y||\cdot|\varphi(y)|\ud y\\
	&+CR^n\int_{ B_{2R}(0)\backslash B_1(0)}|\log|y||\cdot|\varphi(y)|\ud y\\
	\leq &CR^n+CR^n\log(2R)\int_{ B_{2R}(0)\backslash B_1(0)}|\varphi(y)|\ud y\\
	\leq &CR^n\log R
\end{align*}

Finally, the last term can be dealt with by Funibin's theorem.
\begin{align*}
	&\int_{B_R(0)}\int_{ B_{2R}(0)}|\log|x-y||\cdot|\varphi(y)|\ud y\ud x\\
	\leq &\int_{B_{2R}(0)} |\varphi(y)| \ud y\int_{B_{3R}(0)}|\log|z||\ud z\\
	\leq &CR^n\log R.
\end{align*}

Combining these estimates, one has
$$\int_{B_R(0)}|\mathcal{L}(\varphi)(x)|\ud x=O((\log R)\cdot R^n).$$
\end{proof}

We say the conformal  metric $g=e^{2u}|dx|^2$ on $\mr^n$  with finite total Q-curvature is a normal metric if $u$ is a normal solution to \eqref{Q curvature equa}.  To characterize the normal metric,  a volume entropy $\tau(g)$ is introduced in \cite{Li 23 Q-curvature} which is  defined as 
$$\tau(g):=\lim_{R\to\infty}\sup\frac{\log \int_{B_R(0)}e^{nu}\ud x}{\log|B_R(0)|}.$$

\begin{theorem}\label{thm: normal metric iff}
	{\bf (Theorem 1.1 in \cite{Li 23 Q-curvature})} Consider a complete metric $g=e^{2u}|dx|^2$ with finite total Q-curvature on $\mr^n$ where $n\geq 2$ is an even integer. The metric $g$ is normal if and only if $\tau(g)$ is finite. Moreover, if $\tau(g)$ is finite, one has
	$$\tau(g)=1-\frac{2}{(n-1)!|\mathbb{S}^n|}\int_{\mr^n}Q_ge^{nu}\ud x.$$
\end{theorem}

Besides, in \cite{Li 23 Q-curvature}, a geodesic distance $d_g(\cdot, \cdot)$ comparison identity is established which will be used in Section \ref{section:sharp} to show some  metrics are  complete.
\begin{theorem}\label{thm:distance comparison}
	{\bf (Theorem 1.4 in \cite{Li 23 Q-curvature})}
	Consider a conformal metric $g=e^{2u}|dx|^2$ on $\mr^n$   with finite total Q-curvature  where   $n\geq 2$ is an even integer.  Supposing that the metric $g$ is normal, then for each fixed point $p$, there holds
	$$ \lim_{|x|\to\infty}\frac{\log d_g(x,p)}{\log|x-p|}=\left(1-\frac{2}{(n-1)!|\mathbb{S}^n|}\int_{\mr^n}Q_ge^{nu}\ud x\right)^+$$
	where, for a constant $c$, $c^+$ denotes that $c$ if $c\geq 0$ and otherwise $0$.
\end{theorem}

The following Pohozaev-type  inequality is inspired by  the work of Xu (See Theorem 2.1 in \cite{Xu05}). One may also refer to    \cite{Li23 MA} and Lemma 3.1 in \cite{LW}.
\begin{lemma}\label{lem: Pohozaev for non-sign changing}
	 Suppose that $u(x)$ is a  smooth solution to the integral equation
	$$u(x)=\frac{2}{(n-1)!|\mathbb{S}^n|}\int_{\mr^n}\log\frac{|y|}{|x-y|}Q(y)e^{nu(y)}\ud y+C_0$$
where  $C_0$ is a constant, $Qe^{nu}\in L^1(\mr^n)$ and smooth function $Q(x)$ does not change sign near infinity. Then there exists a sequence $R_i\to\infty$ such that
	$$\lim_{i\to \infty}\sup \frac{4}{n!|\mathbb{S}^n|}\int_{B_{R_i}(0)}x\cdot \nabla Q e^{nu}\ud x\leq \alpha_0(\alpha_0-2)$$
	where the notation $\alpha_0$ denotes the normalized  Q-curvature integral 
	$$\alpha_0:=\frac{2}{(n-1)!|\mathbb{S}^n|}\int_{\mr^n}Qe^{nu}\ud x.$$
\end{lemma}
\begin{proof}
Via  a direct computation, one has
	\begin{equation}\label{equ:x,nabla u}
		\langle x,\nabla u\rangle=-\frac{2}{(n-1)!|\mathbb{S}^n|}\int_{\mr^n}\frac{\langle x,x-y\rangle}{|x-y|^2}Q(y)e^{nu(y)}\ud y
	\end{equation}
	Multiplying by $Qe^{nu(x)}$ and integrating over the ball $B_R(0)$ for any $R>0$, we have
	\begin{equation}\label{equ:ingegrate x,nabla u}
		\int_{B_R(0)}Qe^{nu(x)}\left[-\frac{2}{(n-1)!|\mathbb{S}^n|}\int_{\mr^n}\frac{\langle x,x-y\rangle}{|x-y|^2}Q(y)e^{nu(y)}\ud y\right]\ud x=\int_{B_R(0)}Qe^{nu(x)}\langle x,\nabla u(x)\rangle\ud x.
	\end{equation}
	Using  $x=\frac{1}{2}\left((x+y)+(x-y)\right)$, for the left-hand side of \eqref{equ:ingegrate x,nabla u}, one has the following identity
	\begin{align*}
		LHS=&\frac{1}{2}\int_{B_R(0)}Qe^{nu(x)}\left[-\frac{2}{(n-1)!|\mathbb{S}^n|}\int_{\mr^n}Qe^{nu(y)}\ud y\right]\ud x   \\
		&+\frac{1}{2}\int_{B_R(0)}Qe^{nu(x)}\left[-\frac{2}{(n-1)!|\mathbb{S}^n|}\int_{\mr^n}\frac{\langle x+y,x-y\rangle}{|x-y|^2}Qe^{nu(y)}\ud y\right]\ud x.
	\end{align*}
	Now, we deal with  the last term of above equation by changing variables $x$ and $y$.
	\begin{align*}
		&	\int_{B_R(0)}Q(x)e^{nu(x)}\left[\int_{\mr^n}\frac{\langle x+y,x-y\rangle}{|x-y|^2}Q(y)e^{nu(y)}\ud y\right]\ud x\\
		=&\int_{B_R(0)}Q(x)e^{nu(x)}\left[\int_{\mr^n\backslash B_R(0)}\frac{\langle x+y,x-y\rangle}{|x-y|^2}Q(y)e^{nu(y)}\ud y\right]\ud x\\
		=&\int_{ B_{R/2}(0)}Q(x)e^{nu(x)}\left[\int_{\mr^n\backslash B_R(0)}\frac{\langle x+y,x-y\rangle}{|x-y|^2}Q(y)e^{nu(y)}\ud y\right]\ud x\\
		&+\int_{B_R(0)\backslash B_{R/2}(0)}Q(x)e^{nu(x)}\left[\int_{\mr^n\backslash B_{2R}(0)}\frac{\langle x+y,x-y\rangle}{|x-y|^2}Q(y)e^{nu(y)}\ud y\right]\ud x\\
		&+\int_{B_R(0)\backslash B_{R/2}(0)}Q(x)e^{nu(x)}\left[\int_{B_{2R(0)}\backslash B_{R}(0)}\frac{\langle x+y,x-y\rangle}{|x-y|^2}Q(y)e^{nu(y)}\ud y\right]\ud x\\
		=:&I_1(R)+I_2(R)+I_3(R).
	\end{align*}
 Noticing  that for $x\in B_{R/2}(0)$ and $y\in\mr^n\backslash B_R(0)$, one has 
 $$|\frac{\langle x+y,x-y\rangle}{|x-y|^2}|\leq \frac{|x+y|}{|x-y|}\leq 3.$$
  Then one has 
	$$|I_1|\leq 3\int_{ B_{R/2}(0)}|Q(x)|e^{nu(x)}\ud x\int_{\mr^n\backslash B_R(0)}|Q(y)|e^{nu(y)}\ud y. $$
Similarly, there holds
	$$|I_2|\leq 3\int_{ B_R(0)\backslash B_{R/2}(0)}|Q(x)|e^{nu(x)}\ud x\int_{\mr^n\backslash B_{2R}(0)}|Q(y)|e^{nu(y)}\ud y.$$
	Then both $|I_1|$ and $|I_2|$ tend to zero as $R\to \infty$ due to $Qe^{nu}\in L^1(\mr^n)$.
	
	Now,  we only need to  deal with  the term  $I_3$.
	Since $Q$ doesn't change sign near infinity, for $R\gg 1$, one has 
	$$I_3(R)=\int_{B_R(0)\backslash B_{R/2}(0)}Q(x)e^{nu(x)}\left[\int_{B_{2R(0)}\backslash B_{R}(0)}\frac{x^2-y^2}{|x-y|^2}Q(y)e^{nu(y)}\ud y\right]\ud x\leq 0.$$
	
	As for the right-hand side of \eqref{equ:ingegrate x,nabla u}, by using divergence theorem, we have
	\begin{align*}
		RHS=&\frac{1}{n}\int_{B_R(0)}Q(x)\langle x,\nabla e^{nu(x)}\rangle\ud x   \\
		=&-\int_{B_R(0)}\left( Q(x)+\frac{1}{n}\langle x,\nabla Q(x)\rangle\right)e^{nu(x)}\ud x\\
		&+\frac{1}{n}\int_{\partial B_R(0)}Q(x)e^{nu(x)}R\ud \sigma.
	\end{align*}
	Since $Q(x)e^{nu(x)}\in L^1(\mr^n)$, there exist a sequence $R_i\to \infty$ such that
	$$\lim_{i\to\infty}R_i\int_{\partial B_{R_i}(0)}|Q|e^{nu}\ud\sigma=0.$$
	Thus there holds
	\begin{align*}
		&\frac{1}{n}\int_{B_{R_i}(0)}\langle x\cdot \nabla Q\rangle e^{nu}\ud x\\
		=&-\int_{B_{R_i}(0)}Qe^{nu}\ud x+\frac{1}{n}R_i\int_{\partial B_{R_i}(0)}Qe^{nu}\ud \sigma
		+\frac{\alpha_0}{2}\int_{B_{R_i}(0)}Qe^{nu(x)}\ud x \\
		&+I_1(R_i)+I_2(R_i)+I_3(R_i)\\
		\leq &-\int_{B_{R_i}(0)}Qe^{nu}\ud x+\frac{1}{n}R_i\int_{\partial B_{R_i}(0)}Qe^{nu}\ud \sigma
		+\frac{\alpha_0}{2}\int_{B_{R_i}(0)}Qe^{nu(x)}\ud x \\
		&+I_1(R_i)+I_2(R_i)
	\end{align*}
	which yields that
	$$\lim_{i\to \infty}\sup \frac{4}{n!|\mathbb{S}^n|}\int_{B_{R_i}(0)}x\cdot \nabla Q e^{nu}\ud x\leq \alpha_0(\alpha_0-2).$$
	
\end{proof}

\section{Proofs}\label{section:proof}

Throughout our proofs, we will  argue by contradiction and suppose that there exists a complete  conformal metric $g=e^{2u}|dx|^2$ with finite total Q-curvature such that $Q_g=f$ which satisfies
$$(-\Delta)^{\frac{n}{2}}u=fe^{nu}$$
 with $fe^{nu}\in L^1(\mr^n)$. For brevity, set 
$$\beta=\frac{2}{(n-1)!|\mathbb{S}^n|}\int_{\mr^n}fe^{nu}\ud x.$$

{\bf Proof of Theorem \ref{thm: f geq 1}:} 
	
	Since $f\geq 1$ near infinity and $fe^{nu}\in L^1(\mr^n)$, one has $e^{nu}\in L^1(\mr^n)$ which deduces that  $\tau(g)=0.$ With help of Theorem \ref{thm: normal metric iff},  there holds that the metric $g$ is normal and 
	\begin{equation}\label{fe^nu  =1}
		\frac{2}{(n-1)!|\mathbb{S}^n|}\int_{\mr^n}fe^{nu}\ud x=1.
	\end{equation}
Since  $u$ is normal  and $f\geq 1$ near infinity,  for $|y|\gg1$, Lemma \ref{lem: varphi >0} and  \eqref{fe^nu  =1} yield that
\begin{equation}
	u(y)=\mathcal{L}(fe^{nu})(y)+C\geq -\log|y|-C.
\end{equation}
In particular, for $|x|\gg1$ and any $y\in B_{|x|/2}(x)$, there holds
\begin{equation}\label{u geq -alpha _0}
	u(y)\geq -\log(\frac{3}{2}|x|)-C=-\log|x|-C. 
\end{equation}
For $|x|\gg1$, with help of Jensen's inequality, the estimate \eqref{u geq -alpha _0} and the fact $f\geq 1$ near infinity show that 
\begin{align*}
	\int_{B_{\frac{|x|}{2}}(x)}fe^{nu}\ud y\geq &	\int_{B_{\frac{|x|}{2}}(x)}e^{nu}\ud y\\
	\geq &|B_{\frac{|x|}{2}}(x)|\exp\left(\frac{1}{|B_{\frac{|x|}{2}}(x)|}\int_{B_{\frac{|x|}{2}}(x)}nu(y)\ud y\right)\\
	\geq &C|x|^{n}\cdot |x|^{-n}\\
	\geq &C
\end{align*}
which contradicts to  $fe^{nu}\in L^1(\mr^n)$.  Thus we finish our proof.

\vspace{1em}
{\bf Proof of Theorem \ref{thm:xnabla log f geq -1}:}	Due to the condition \eqref{condition for f}, Lemma \ref{lem: f geq |x|^s} yields that
	\begin{equation}\label{ f decay |x|-n/2}
		f(x)\geq C|x|^{-\frac{n}{2}},\;  |x|>1.
	\end{equation}
	Since $fe^{nu}\in L^1(\mr^n)$, the estimate \eqref{ f decay |x|-n/2} shows that for $R>1$, 
	\begin{align*}
		\int_{B_R(0)}e^{nu}\ud x \leq & \int_{B_1(0)}e^{nu}\ud x+\int_{B_R(0)\backslash B_1(0)}e^{nu}\ud x\\
		\leq &C+C\int_{B_R(0)\backslash B_1(0)}f(x)|x|^{\frac{n}{2}}e^{nu(x)}\ud x\\
		\leq &C+CR^{\frac{n}{2}}\int_{B_R(0)\backslash B_1(0)}f(x)e^{nu(x)}\ud x\\
			\leq &CR^{\frac{n}{2}}
	\end{align*}
which yields that
$$\tau(g)\leq \frac{1}{2}.$$
 Theorem \ref{thm: normal metric iff} shows that  $u$ is normal and 
	\begin{equation}\label{complete yield alpha leq 1}
		\beta =1-\tau(g)\leq 1
	\end{equation}
where we have used the fact $\tau(g)\geq 0$.
By the definition of $\tau(g)$, one has
\begin{equation}\label{equ:tau(g)geq0}
	\tau(g)\geq \lim_{R\to\infty}\inf\frac{\log\int_{B_1(0)}e^{nu}\ud x}{\log|B_R(0)|}=0.
\end{equation}
	With help of  the condition  \eqref{condition for f}, one has 
	\begin{align*}
		&\frac{4}{n!|\mathbb{S}^n|}\int_{B_{R}(0)}x\cdot \nabla f e^{nu}\ud x\\
		\geq & \frac{4}{n!|\mathbb{S}^n|}\int_{B_{R}(0)} -\frac{n}{2}fe^{nu}\ud x\\
		=&-\frac{2}{(n-1)!|\mathbb{S}^n|}\int_{B_{R}(0)} fe^{nu}\ud x.
	\end{align*}
	Taking advantage of Lemma \ref{lem: Pohozaev for non-sign changing} and choosing a suitable sequence, there holds
	$$\beta(\beta-2)\geq -\beta$$
	which yields that $$\beta \geq 1$$ where we have used the fact $\beta >0$ since $f>0$. If the equality holds, one must have
	$$x\cdot \nabla f(x)=-\frac{n}{2}f(x), a.e. $$
	which is impossible by  choosing  sufficiently small $\delta>0$ such that   $$x\cdot \nabla f(x)+\frac{n}{2}f(x)>0$$ for $x\in B_\delta(0)$. Hence we obtain that  $$\beta >1$$ which contradictions to \eqref{complete yield alpha leq 1}.
	
Thus the  proof is complete.

\vspace{1em}
{\bf Proof of Theorem \ref{thm: f leq 0}:} Based on  the assumption \eqref{f leq -|x|^-n},  there exists $R_1>0$ such that for $|x|\geq R_1$,
\begin{equation}\label{|x| geq R_1}
	|f(x)|\geq C|x|^{-n}.
\end{equation}
For $R>R_1+1$, there holds
\begin{align*}
	\int_{B_R(0)}e^{nu}\ud x=& \int_{B_{R_1}(0)}e^{nu}\ud x+\int_{ B_R(0)\backslash B_{R_1}(0)}e^{nu}\ud x\\
	\leq &C-C\int_{ B_R(0)\backslash B_{R_1}(0)}f(x)|x|^{n}e^{nu}\ud x\\
	\leq &C-CR^{n}\int_{ B_R(0)\backslash B_{R_1}(0)}f(x)e^{nu}\ud x\\
	\leq &CR^{n}
\end{align*}
which yields that $\tau(g)$ is finite. Making use of Theorem \ref{thm: normal metric iff}, we show that $u(x)$ is normal  satisfying
$$u(x)=\mathcal{L}(fe^{nu})+C.$$
With help of Jensen's inequality and Lemma \ref{lem:B_r_0|x|}, for $|x|\gg1$, one has
\begin{align*}
	&\int_{B_{\frac{|x|}{2}}(x)}|f(y)|e^{nu(y)}\ud y\\
	\geq &C\int_{B_{\frac{|x|}{2}}(x)}|y|^{-n}e^{nu(y)}\ud y\\
	\geq &C|x|^{-n}\int_{B_{\frac{|x|}{2}}(x)}e^{n\mathcal{L}(fe^{nu})(y)}\ud y\\
	\geq &C|x|^{-n}|B_{\frac{|x|}{2}}(x)|\exp\left(\frac{1}{|B_{\frac{|x|}{2}}(x)|}\int_{B_{\frac{|x|}{2}}(x)}n\mathcal{L}(fe^{nu})(y)\ud y\right)\\
	\geq &C|x|^{-n}\cdot |x|^n\cdot|x|^{-n\beta+o(1)}.
\end{align*}
Then $fe^{nu}\in L^1(\mr^n)$  deduces that  
\begin{equation}\label{beta geq 0}
	\beta\geq 0.
\end{equation}
However, since $f\leq 0$ satisfies \eqref{f leq -|x|^-n}, we must have $\beta <0$ which contradicts to \eqref{beta geq 0}.

\vspace{1em}
{\bf Proof of Theorem \ref{thm: f < 0 near infinity}:} Due to assumption \eqref{f leq- |x|^s} and $fe^{nu}\in L^1(\mr^n)$, similar to the proof of Theorem \ref{thm: f leq 0}, one has $\tau(g)$ is finite and then $u(x)$ is  normal.

With help of Jensen's inequality and Lemma \ref{lem:B_r_0|x|}, for $|x|\gg1$, there holds
\begin{align*}
	&\int_{B_{\frac{|x|}{2}}(x)}|f(y)|e^{nu(y)}\ud y\\
	\geq &C\int_{B_{\frac{|x|}{2}}(x)}|y|^{s}e^{nu(y)}\ud y\\
	\geq &C|x|^{s}\int_{B_{\frac{|x|}{2}}(x)}e^{n\mathcal{L}(fe^{nu})(y)}\ud y\\
	\geq &C|x|^{s}|B_{\frac{|x|}{2}}(x)|\exp\left(\frac{1}{|B_{\frac{|x|}{2}}(x)|}\int_{B_{\frac{|x|}{2}}(x)}n\mathcal{L}(fe^{nu})(y)\ud y\right)\\
	\geq &C|x|^{s}\cdot |x|^n\cdot|x|^{-n\beta+o(1)}.
\end{align*}
which  yields that
$$\beta\geq 1+\frac{s}{n}>1.$$
However, Theorem  \ref{thm: normal metric iff} deduces that $\beta \leq 1$  due to \eqref{equ:tau(g)geq0} which is the desired contradiction.

\section{Sharp decay  rate for $f$}\label{section:sharp}

With help of Lemma \ref{lem: f geq |x|^s}, the condition \eqref{condition for f} deduces that 
$$f(x)\geq C|x|^{-\frac{n}{2}},\;|x|>1.$$
For $n=2$, in  \cite{Cheng-Lin Pisa}, \cite{HuTr}, \cite{McOwen} and \cite{Tro}, their results ensure  that
for a positive function $f(x)$ satisfying 
$$f(x)\leq C|x|^{-l},|x|\gg1, \;l>0, $$
the solutions to  equation \eqref{NP for R^n} exist.  To serve our aim, a baby version of McOwen's result \cite{McOwen} can be  stated as follows.
\begin{theorem}\label{thm: McO}
	{\bf (McOwne \cite{McOwen})} Given a smooth function $f(x)$ which is  positive somewhere and  satisfies 
	\begin{equation}
		 f(x)= O(|x|^{-l})
	\end{equation}
	where $l>0$,  $|x|\gg1$ and $C$ is a positive  constant.  Then for any $\alpha\in (\max\{0,2-l\}, 2)$, there exist a solution $u(x)$ to the following equation
	$$-\Delta u=fe^{2u},\;\mathrm{on}\; \mr^2$$
	satisfying  $$\alpha=\frac{1}{2\pi}\int_{\mr^2}fe^{2u}\ud x.$$
\end{theorem}

\begin{theorem}\label{thm:exist for l>1}
	Consider a positive and smooth function $f$ on $\mr^2$  satisfying 
	\begin{equation}\label{f |x|-l decay}
		C^{-1}|x|^{-l}\leq f(x)\leq C|x|^{-l}
	\end{equation}
where $l>1$, $|x|\gg1$ and $C$ is a positive constant. There exists a complete metric $g=e^{2u}|dx|^2$ on $\mr^2$ with finite total curvature such that its  Gaussian curvature $K_g=f$.
\end{theorem}
\begin{proof}
	With help of Theorem \ref{thm: McO} and the fact $l>1$, we can find a solution $u(x)$ to the equation
	\begin{equation}\label{Dleta u=fe^2u}
		-\Delta u=fe^{2u}
	\end{equation}
	satisfying 
	\begin{equation}\label{total curvature <1}
		\frac{1}{2\pi}\int_{\mr^2}fe^{2u}\ud x<1.
	\end{equation}
Based on the condition  \eqref{f |x|-l decay} and \eqref{total curvature <1}, for $R\gg1$, one has 
\begin{equation}\label{volume finite growth}
	\int_{B_R(0)}e^{2u}\ud x\leq C R^l.
\end{equation}

With help of the fact $\frac{2}{l+4}u^+\leq e^{\frac{2}{l+4}u}$ and H\"older's inequality, one has
\begin{align*}
	\int_{B_R(0)}u^+(x)\ud x\leq & \frac{l+4}{2}\int_{B_R(0)}e^{\frac{2}{l+4}u}\ud x\\
	\leq &\frac{l+4}{2}\left(\int_{B_R(0)}e^{2u}\ud x\right)^{\frac{1}{l+4}}|B_R(0)|^{\frac{l+3}{l+4}}\\
	\leq &CR^{\frac{l}{l+4}}R^{\frac{2(l+3)}{l+4}}\\
	=&CR^{3-\frac{6}{l+4}}
\end{align*}
which yields that
\begin{equation}\label{int u^+}
	\frac{1}{|B_R(0)|}\int_{B_R(0)}u^+\ud x=o(R).
\end{equation}
Set
$$v(x):=\frac{1}{2\pi}\int_{\mr^2}\log\frac{|y|}{|x-y|}f(y)e^{2u(y)}\ud y$$
and 
$$P(x):=u-v$$
satisfying
\begin{equation}\label{harmonic P}
	\Delta P=0.
\end{equation}
Making use of Lemma \ref{lem: Delta Lf},  for $R\gg1$, there holds
\begin{equation}\label{B_Rv}
	\frac{1}{|B_R(0)|}\int_{B_R(0)}|v(x)|\ud x=O(\log R)=o(R).
\end{equation}
Combing the estimate \eqref{int u^+} with \eqref{B_Rv}, one has
\begin{equation}\label{int_B_R P}
	\frac{1}{|B_R(0)|}\int_{B_R(0)}P^+(x)\ud x\leq \frac{1}{|B_R(0)|}\int_{B_R(0)}\left(u^+(x)+|v(x)|\right)\ud x=o(R).
\end{equation}
With help of Liouville's theorem,  \eqref{harmonic P} and \eqref{int_B_R P} deduce  that 
$$P(x)\equiv C.$$
Thus $u(x)$ is a normal solution to \eqref{Dleta u=fe^2u}. Making use of Theorem \ref{thm:distance comparison},  there holds 
\begin{equation}\label{distance}
	\lim_{|x|\to\infty}\frac{\log d_g(x,p)}{\log|x-p|}=1-\frac{1}{2\pi}\int_{\mr^2}fe^{2u}\ud x>0
\end{equation}
which yields that the conformal metric $g$ is complete(See Theorem 5.7.1 in \cite{Petersen}). Finally, we finish our proof.
\end{proof}

 For $n\geq 4$, a result  analogous to Theorem \ref{thm: McO} might  still hold  by  using the method taken in \cite{McOwen}  or  \cite{Chang-Chen}.  A crucial ingredient is 
 Proposition 1  in \cite{McOwen} which  a singular type Moser-Trudinger inequality which is also established in Theorem 6 of \cite{Tro}. For higher dimensional cases, such a singular type Adams-Moser–Trudinger inequality has been established by Theorem 4.6 in \cite{F-M}, Theorem 2.4 in \cite{JYY}.  With help of such inequality,  one may consider a suitable functional and  find its minimizer which a normal solution with finite total Q-curvature less than $\frac{(n-1)!|\mathbb{S}^n|}{2}$. Finally, making use of Theorem \ref{thm:distance comparison}, one may show that such normal metric  is  complete. However, this is just the general idea, actually realizing  it and making it clear is not an easy task. For convenience, we leave it as a question.
 \begin{question}
 	 Given a  smooth $f(x)$ positive somewhere on $\mr^n$ where $n\geq 2$ is an even integer and $f(x)$ satisfies 
 	$$
 	f(x)=O(|x|^{-s}),\; s>\frac{n}{2},
 	$$
 	then there is a complete   conformal metric $g=e^{2u}|dx|^2$ on $\mr^n$  with finite total Q-curvature  such that its  Q-curvature $Q_g=f(x)$.
 	\end{question}
 \begin{remark}
 	After submitting this work, this question has been answered by the author joint with Biao Ma in \cite{LM}.
 \end{remark}


\begin{thebibliography}{99}
	\bibitem{Aviles}
 Aviles, P.:  Conformal complete metrics with prescribed nonnegative Gaussian curvature in $R^2$.  \emph{Invent. Math.} \textbf{83} (1986), no.~3, 519--544. 
\bibitem{Brendle}
Brendle, S.: Convergence of the Q-curvature flow on $S^4$.  \emph{Adv. Math.} \textbf{ 205} (2006), no.~1, 1--32.
\bibitem{Chang-Chen}
Chang, S.-Y.~A. and  Chen, W.:  A note on a class of higher order conformally covariant equations.  \emph{Discrete Contin. Dyn. Syst.}  \textbf{7} (2001), no.~2, 275–281.
\bibitem{CQY}
Chang, S.-Y.~A. ,  Qing, J.  and Yang, P.~C. :  On the Chern-Gauss-Bonnet integral for conformal metrics on $\mathbb{R}^4$.
\emph{Duke Math. J.}  \textbf{103} (2000), no.~3, 523--544. 
\bibitem{CQY2}
Chang, S.-Y.~A. ,  Qing, J.  and Yang, P.~C. :   Compactification of a class of conformally flat 4-manifold. \emph{Invent. Math.}  \textbf{142} (2000), no.~1, 65--93.
\bibitem{CY Acta}
Chang, S.-Y.~A. and Yang, P.~C.:  Prescribing Gaussian curvature on $S^2$.   \emph{Acta Math.} \textbf{159} (1987), no.~3-4, 215--259. 
\bibitem{CY JDG}
Chang, S.-Y.~A. and Yang, P.~C.: Conformal deformation of metrics on $S^2$.  \emph{J. Differential Geom.} \textbf{27} (1988), no.~2, 259--296.
\bibitem{CY 95 Ann}
Chang, S.-Y.~A. and Yang, P.~C.:  Extremal metrics of zeta function determinants on 4-manifolds.  \emph{Ann. of
	Math.} \textbf{142} (1995), 171–212.
\bibitem{CYZ}
Chen, H.,   Ye, D.  and  Zhou, F.: On Gaussian curvature equation in $R^2$ with prescribed nonpositive curvature.   \emph{Discrete Contin. Dyn. Syst.} \textbf{40} (2020), no.~6, 3201--3214.
\bibitem{CL 91 Duke}
Chen, W. and  Li, C.: Classification of solutions of some nonlinear elliptic equations.   \emph{Duke Math. J.} \textbf{63} (1991), no.~3, 615--622.
\bibitem{CL 93}
Chen, W. and  Li, C.: Qualitative properties of solutions to some nonlinear elliptic equations in $R^2$.   \emph{Duke Math. J. } \textbf{71} (1993), no.~2, 427--439.
\bibitem{Cheng-Lin MA}
Cheng, K. -S. and   Lin, C.-S.: On the asymptotic behavior of solutions of the conformal Gaussian curvature equations in $R^2$.   \emph{Math. Ann.} \textbf{308} (1997), no.~1, 119--139.
\bibitem{Cheng-Lin Pisa}
Cheng, K. -S. and   Lin, C.-S.:  Compactness of conformal metrics with positive Gaussian curvature in $R^2$.  \emph{Ann. Scuola Norm. Sup. Pisa Cl. Sci. (4)}  \textbf{26} (1998), no.~1, 31--45.
\bibitem{CV}
Cohn-Vossen, S.:   K\"urzeste Wege und Totalkr\"ummung auf Fl\"achen.   \emph{Compositio Math.}  \textbf{2} (1935), 69--133.
\bibitem{Fa}
Fang, H.: On a conformal Gauss–Bonnet–Chern inequality for LCF manifolds and related topics.   \emph{Calc. Var. Partial Differential Equations}  \textbf{23} (2005), no.~4,  469--496.
\bibitem{F-M}
Fang, H.  and  Ma, B.:  Constant Q-curvature metrics on conic 4-manifolds.  \emph{Adv. Calc. Var.} \textbf{15} (2022), no.~2, 235--264.
\bibitem{Hu}
Huber, A.: On subharmonic functions and differential geometry in the large.    \emph{Comment. Math. Helv.}  \textbf{32} (1957), 13--72.
\bibitem{HuTr}
Hulin, D. and  Troyanov, M.:  Prescribing curvature on open surfaces.   \emph{Math. Ann.}  \textbf{293} (1992), no.~2, 277--315.
\bibitem{HMM}
Hyder, A.,  Mancini, G.  and  Martinazzi, L.: Local and nonlocal singular Liouville equations in Euclidean spaces.  \emph{ Int. Math. Res. Not. IMRN} 2021, no.~15, 11393--11425.
\bibitem{JYY}
Jevnikar, A.,   Yannick, S. and   Yang, W.: Prescribing Q-curvature on even-dimensional manifolds with conical singularities.  \emph{Rev. Mat. Iberoam.} \textbf{41}  (2025), no.~1, 1--28.
\bibitem{JLX}
Jin,  T., Li, Y. and Xiong, J.: On a fractional Nirenberg problem, part I: blow up analysis and compactness of solutions.  \emph{J. Eur. Math. Soc. (JEMS)} \textbf{16} (2014), no.~6, 1111--1171. 
\bibitem{KY}
Kalka, M. and  Yang, D.:  On conformal deformation of nonpositive curvature on noncompact surfaces.   \emph{Duke Math. J.} \textbf{72} (1993), no.~2, 405--430.
\bibitem{KZ 74 Compact}
Kazdan, J.  and  Warner, F. : Curvature functions for compact 2-manifolds.  \emph{Ann. of Math.(2)}  \textbf{99} (1974), 14--47. 
\bibitem{KZ 74 open}
Kazdan, J.  and  Warner, F. :Curvature functions for open 2-manifolds.   \emph{Ann. of Math. (2)} \textbf{99} (1974), 203--219. 
\bibitem{Li23 MA}
Li, M.:  A Liouville-type theorem in conformally invariant equations.   \emph{Math. Ann.} \textbf{389} (2024), no.~3, 2499--2517.
\bibitem{Li 23 Q-curvature}
Li, M.:   The total Q-curvature, volume entropy and polynomial growth polyharmonic functions.  \emph{Adv. Math. } \textbf{450} (2024), Paper No. 109768, 43 p.
\bibitem{LM}
Li, M. and  Ma, B.:	Existence of complete conformal metrics on $\mathbb{R}^n$ with prescribed Q-curvature,   priprint 2025, arxiv:2503.23689. 
\bibitem{LW}
Li, M. and  Wei, J.: Higher order Bol's inequlity and its applications, 	priprint 2023, arxiv:2308.11388.
\bibitem{Lin}
Lin, C.-S.:  A classification of solutions of a conformally invariant fourth order equation in $R^n$. \emph{Comment. Math. Helv. } \textbf{73} (1998), no.~2, 206--231.
\bibitem{Lu-Wang}
Lu, Z. and  Wang, Y.:	On locally conformally flat manifolds with finite total Q-curvature.  
\emph{Calc. Var. Partial Differential Equations} \textbf{56} (2017), no.~4, Paper No. 98, 24 pp. 
\bibitem{MalStru}
Malchiodi, A. and   Struwe, M.:  Q-curvature flow on $S^4$.  \emph{J. Differential Geom.} \textbf{73} (2006), no.~1, 1--44.
\bibitem{McOwen}
McOwen, R.: Conformal metrics in $R^2$ with prescribed Gaussian curvature and positive total curvature.   \emph{Indiana Univ. Math. J.} \textbf{34} (1985), no.~1, 97--104.
\bibitem{Mar MZ}
Martinazzi, L.:  Classification of solutions to the higher order Liouville's equation on $\mr^{2m}$.    \emph{Math. Z.} \textbf{263} (2009), no.~2, 307--329. 
\bibitem{Martinazzi AANL}
Martinazzi, L.:  Conformal metrics on $\mr^{2m}$ with constant Q-curvature.   \emph{Atti Accad. Naz. Lincei Rend. Lincei Mat. Appl.} \textbf{19} (2008), no.~4, 279--292.
\bibitem{NX}
Ndiaye, C. and  Xiao, J.: 
An upper bound of the total Q-curvature and its isoperimetric deficit for higher-dimensional conformal Euclidean metrics.   \emph{Calc. Var. Partial Differential Equations}  \textbf{38} (2010),  no.~1-2, 1--27.
\bibitem{Ni 82}
Ni, W.-M.:
On the elliptic equation $\Delta u+K(x)e^{2u}=0$ and conformal metrics with prescribed Gaussian curvatures. 	\emph{Invent. Math.} \textbf{66} (1982), no.~2, 343--352. 
\bibitem{Petersen}
Petersen,  P.: \emph{Riemannian geometry. Third edition}.  Graduate Texts in Mathematics, 171. Springer, Cham, 2016. xviii+499 pp. ISBN: 978-3-319-26652-7; 978-3-319-26654-1.
\bibitem{Sat}
Sattinger, D.:
Conformal metrics in $R^2$ with prescribed curvature.   
\emph{Indiana Univ. Math. J. } \textbf{22} (1972/73), 1--4. 
\bibitem{Struwe}
Struwe, M.:  "Bubbling'' of the prescribed curvature flow on the torus.   \emph{J. Eur. Math. Soc. (JEMS)} \textbf{22} (2020), no.~10, 3223--3262. 
\bibitem{Tro}
Troyanov, M.: Prescribing curvature on compact surfaces with conical singularities.   \emph{Trans. Amer. Math. Soc.} \textbf{324} (1991), no.~2, 793--821.
\bibitem{Wang Yi}
Wang,  Y.: The isoperimetric inequality and Q-curvature.   \emph{Adv. Math.} \textbf{281} (2015), 823--844. 
\bibitem{WX}
Wei, J. and  Xu, X.: Classification of solutions of higher order conformally invariant equations.    \emph{Math. Ann.} \textbf{313} (1999), no.~2, 207--228.
\bibitem{WX 09 JFA}
Wei, J. and  Xu, X.:  Prescribing Q-curvature problem on $S^n$.  \emph{J. Funct. Anal.} \textbf{257} (2009), no.~7, 1995--2023.
\bibitem{Xu05}
Xu, X.:  Uniqueness and non-existence theorems for conformally invariant equations.  \emph{J. Funct. Anal.} \textbf{222} (2005), no.~1, 1--28.

\end{thebibliography}
\end{document}